\documentclass[aop]{imsart}

\RequirePackage{amsthm,amsmath,amsfonts,amssymb}
\RequirePackage[numbers]{natbib}
\RequirePackage[colorlinks,citecolor=blue,urlcolor=blue,pdfborder={0 0 0}]{hyperref}
\RequirePackage{graphicx}
\usepackage{arydshln} 

\startlocaldefs
\numberwithin{equation}{section}
\theoremstyle{plain}

\newtheorem{theorem}{Theorem}[section]
\newtheorem{lemma}[theorem]{Lemma}
\theoremstyle{definition}
\newtheorem{definition}[theorem]{Definition}

\endlocaldefs

\begin{document}

\begin{frontmatter}
\title{THEORY OF UNCERTAIN PROBABILITY:\\ 
CAN WE DERIVE THE PROBABILITY DENSITY FUNCTION OF UNCERTAIN RANDOM EXPERIMENTS 
WITH CONTINUOUSLY CHANGING CONDITIONS?}
\runtitle{THEORY OF UNCERTAIN PROBABILITY}

\begin{aug}
\author[A]{\fnms{Xiaolin}~\snm{Gong}\ead[label=e1]{gxiaolin@sdu.edu.cn}}
\address[A]{Shandong University\printead[presep={,\ }]{e1}}
\end{aug}

\begin{abstract}
This paper aims to explore the formation mechanism of probability distribution in situations where the differences among random experiments are distinguishable, and these differences continue to evolve along with the dynamic changes in conditions and their mechanisms of action. To this end, we are motivated to devise a new theoretical system---theory of uncertain probability (TUP) with Kolmogorov’s system and nonlinear theories as special cases. TUP develops a novel model that integrates probability and uncertainty as well as the known and unknown to more accurately depict numerous typical random phenomena under more realistic assumptions, and thus provides appropriate tools for greater variety of real needs. It also allows for pioneering interpretation of the causal mechanisms underlying many important distributional characteristics and incorporation of pathwise property to distribution model.
\end{abstract}

\begin{keyword}[class=MSC]
\kwd[Primary ]{60E05}
\kwd{62H05}
\kwd[; secondary ]{60A86}
\end{keyword}

\begin{keyword}
\kwd{probability distribution}
\kwd{uncertainty}
\kwd{distributional characteristics}
\kwd{random experiment}
\end{keyword}

\end{frontmatter}


\section{Introduction}
The following is a short justification for the development of a new probability theory to provide suitable tools for the broader range of needs of the natural and social sciences. According to William Feller \cite{R7}, idealization is indispensable for devising the modern theory of probability and also a standard procedure in other mathematical disciplines. Specifically, the sample space together with the probability distribution assigned in it defines the idealized experiment. And this idealized experiment is described by Kolmogorov \cite{R14} to take place with a complex of conditions, which allows any number of repetitions. Although this idealization practice does not affect the application of the abstract models in many diverse fields, it does hinder their applicability in numerous very important fields especially those of the social sciences. Moreover, nonlinear theories intended to deal with uncertainty develop models with arbitrarily assumed and thus controversial bounds, rather than delving deeply into the underlying issues of actual experiments and their conditions. Although the above theories have already been used in the social sciences, controversy has never ceased. And unsatisfactory, unpleasant empirical evidence, with failure in risk management leading to the 2007-2008 global financial crisis as the most prominent instance, has been constantly supporting the opposing views. The real conditions of social sciences obviously do not allow for any number of repetitions or varying distributions that are theoretically or arbitrarily assigned (e.g., \cite{R3}). At the same time, we are aware that similar theoretical inefficiencies exist philosophically and in practice in the natural sciences.

As well-known from extensive literature (e.g., \cite{R10, R18}), revealing the complexity of deterministic phenomena and the statistical determinacy of accidental phenomena is a main research interest and goal of Kolmogorov during his later years with a basic idea that there is actually no true boundary between the ordered realm and the accidental realm, and the mathematical world is essentially an indivisible whole. Specifically, as to the frequently observed characteristics of distribution, pathwise properties, and features of data series, do they occur by chance or have some degree of determinacy? Furthermore, if both of the two contradictory elements, the deterministic and accidental, act together, then how can they be combined into a single theory of probability?

In sum, we are motivated to address the following fundamental problem: in practice, the conditions in many random experiments are non-repetitive. Moreover, the most complex scenario occurs when the experimental conditions and their mechanisms of action continuously change.  However, the experimental results still show some probability‐-like patterns. Such issues require analysis, but all current theories and models are found to be inadequate. Changes in experimental conditions and their mechanisms of action are the source of probabilistic uncertainty. The solution of the problem is to figure out how to integrate probability with uncertainty, merge known with unknown, and combine ordered realm with accidental realm.

\section{Preliminaries}

We believe that most of us may agree that an appropriate degree of idealization is necessary, useful, and acceptable, but the overload of it is hard to accept. Thus, with the new theory of probability, we do not analyze purely conceptual random experiments. Depending on how different they are from each other, experiments can be divided into the following 4 categories: 

\begin{longlist}
\item
Indistinguishable random experiments

It is impossible to distinguish among the experiments or the cost of distinguishing is too high, not worthwhile, or unnecessary. Examples are: drawing balls from an urn, tossing a die with really hard texture, observing pollen under diffusion at almost constant temperature, etc. It has proved that Kolmogorov system of probability (henceforth denoted by KSP) is applicable and effective in such cases.

\item
Marginally distinguishable random experiments

Without deep philosophical discussions, we are aware that repetitive situations are not typical in real life and there are often somewhat differences among experiments. According to the degree of precision required for the probability distribution, the purpose of the study and the application objectives, we may choose to ignore these marginally distinguishable differences or treat them as the following category of random experiments.

\item
Partially distinguishable random experiments (PDREs) \footnote{Details about PDREs are given in 4.1..}

In both the natural and social sciences, there are a lot of scenarios with partially distinguishable experiments. Especially, as studies shifts to traditional social science subjects like decision-making, \footnote{Decision-based actions and their results, such as purchases and the quantities of purchases made, are not eligible for study as random experiments because they appear to be essentially identical and are easily and mistakenly classified as indistinguishable or marginally distinguishable experiments. And this is one of the reasons why KSP is applied so rashly to the social sciences and leads to so many unsatisfactory applications.}  we can certainly identify the differences and even the patterns of the differences among experiments where, as Ludvig von Mises (\cite{R21}) who is the founder of the "neo-Austrian School" of economics noted, teleology is the prominent rule. \footnote{What is especially special about scenarios in social science is that we can not only explore individual differences by general methods such as observation and analysis, but also obtain more detailed information by means of communication, which is not normally possible in natural science. And just because we can get more information about each individual, and therefore can more easily and clearly distinguish one from another, repetitive experiments are more impossible.} And Frank Knight (\cite{R11}), the founder of Chicago school of economics, notes that this form of probability is involved in the greatest logical difficulties of all. In the natural sciences, except for those with conditions strictly controlled in laboratories, \footnote{Similar exception includes the exceedingly narrow domain of experience connected with games of chance.} it’s not difficult to recognize some fundamental and/or influencing conditions/factors that differentiate the experiments. Thus, the applicability of KSP is dampened in scenarios either dominated by the social laws or by many natural laws. Furthermore, with knowledge accumulation, we know more and more about what the specific causes are, how they make the conditions non-repetitive which further affect the observed outcomes of the experiments and/or induce the probabilities far from arbitrarily assumed distributions, the bounds of expectations, etc. Hence, the applicability of the nonlinear theories is also limited although they have tried to take into account distinguishable differences by assuming different distributions. 

However, needless to say in the natural sciences and even in social settings, the experiments are at most partially distinguishable since humans have a dual nature of biology and sociality. Among many convincing instances of the biological attributes of human, are experiments like selecting a random sample of 50 people and observing the number of left-handers. We know that KSP is very effective in these instances. Especially, biological basis cannot be divorced from any mental, psychological, emotional, and even spiritual activities which are distinguishable to different extents especially through the help of communication. Then, the problem becomes how should we deal with the distinguishable part together with the indistinguishable part in an indivisible whole theory. \footnote{We will not discuss other interpretations of probability, such as subjective probability, but it can be seen later that our model can deal with subjective probability in a more objective way.}

\item
Fully distinguishable experiments

As is widely recognized, other theories, not probability theories, should be used to derive the laws in these cases with fully distinguishable experiments.

Therefore, the new probability theory is based on the analysis of the most complex but general scenarios with partially distinguishable random experiments. And as shown below, the new theory could easily degenerate into KSP, it can express the subadditivity condition behind ambiguity aversion in the theory of capacities in a more comprehensive way, the nonlinear expectation theories actually focus on its special intermediate states, and the nonlinear probability theories constitute (discontinuous) partial modeling of it. Moreover, the new theoretical framework can effectively integrate probability and uncertainty into a coherent whole theory without making implausible assumptions and is therefore referred to as theory of uncertain probability (henceforth denoted by TUP).

\end{longlist}

\section{Definitions and notation}

\begin{definition}
As the basic notion in TUP, PDRE has outcomes that are even more difficult to predict in advance since PDREs are performed under specifically varying conditions, but are nevertheless still subject to analysis.
\end{definition}

\begin{definition}
If the set of possible outcomes of a $PDRE_n$ is a sample space $\Phi_{n}$, then the set of all possible outcomes of the relevant $PDRE_{n}$, $n=1, 2, \ldots$ is the union of sample spaces $\Phi_{n}$, $n=1, 2, \ldots$, arising individually in an uncertain manner, denoted by $\Pi$, called an uncertain sample space and equal to

\begin{equation}
\begin{aligned}
\bigcup_{n=1}^{\infty}\Phi_{n}.
\end{aligned}
\end{equation}
\end{definition}

\begin{definition}
Let $\boldsymbol{\pi}$ be an uncertain sample point of $\Pi$. We say that a set of uncertain sample points is an uncertain random event denoted by capital letters such as $\mathfrak{A}$, $\mathfrak{B}$, $\ldots$.

Specifically, in the continuous setting, \footnote{It is clear that all the following theoretical derivations in the paper also apply to scenarios with $\Phi_{n}$, $n=1,2, \ldots$, and $\Pi=\bigcup_{n=1}^{\infty}\Phi_{n}$.} if the set of possible outcomes of $PDRE_{\gamma}$ is sample space $\Phi_\gamma$, then the set of all possible outcomes of the relevant PDREs is 

\begin{equation}
\begin{aligned}
\Pi=\bigcup_{\gamma\in\Lambda}\Phi_{\gamma}=\left\{\varphi|\exists\gamma\in{\Lambda}\ \text{such that}\ {\varphi\in\Phi_{\gamma}}\right\}.
\end{aligned}
\end{equation}

Where $\Lambda$ is an index set, $\left\{\Phi_{\gamma}:\gamma\in\Lambda\right\}$ is a family of sets, and $\varphi$ is an outcome in $\Phi_{\gamma}$.

It is obvious that if $\Phi_{\gamma}=\mathcal{R}$, $\Pi=\mathcal{R}$, it seems that differentiating among sample spaces is not necessary. Nevertheless, it is indispensable for all the other scenarios as well as for the universality of the theory and the consistency of the theoretical framework. 

Moreover, let $V_{\gamma}$ be the random variable on $\Phi_{\gamma}$ and designated as random variable of uncertain experiment. 

\begin{equation}
\begin{aligned}
    & F_{V_{\gamma}}(\mathrm{v_{\gamma}})=\int_{-\infty}^{\mathrm{v_{\gamma}}}{f_{V_{\gamma}}\left(v_{\gamma}\right)}dv_{\gamma} \\
    & \mathcal{P}_{V_{\gamma}}(v_{\gamma})=f_{V_{\gamma}}(v_{\gamma})
\end{aligned}
\end{equation}

$F_{V_{\gamma}}$ and $f_{V_{\gamma}}$ are the distribution function and the probability density function of $V_{\gamma}$, respectively. Only in KSP scenarios, are  $V_{\gamma}$, $F_{V_{\gamma}}$ and $f_{V_{\gamma}}$ identical across different $\gamma$.
\end{definition}

\begin{definition}
As the set of uncertain random events, $\mathbb{F}$ is a $\sigma$-field composed of all subsets of $\Pi$. 
\end{definition}

\begin{definition}
Given a measurable space $\left(\Pi,\mathbb{F}\right)$, the set function $\mathbb{P}$ on $\mathbb{F}$ is a measure with

\begin{longlist}
\item $\mathbb{P}\left(\mathfrak{A}\right)\geq\mathbb{P}\left(\emptyset\right)=0$ for all $\mathfrak{A}\in\mathbb{F}$;
\item 
$\mathbb{P}$ is $\sigma$-additive;
\item 
$\mathbb{P}=\mathcal{P}_{\mathbb{X}}\left(\mathcal{P}_{V_{\gamma}}\left(v_{\gamma}\right)\right)$. 
\end{longlist}

It will be proved that $\mathbb{P}\left(\Pi\right)=1$, and  is referred to as a measure of uncertain probability.   
\end{definition} 

\begin{definition} 
The triple $\left(\Pi,\mathbb{F},\mathbb{P}\right)$ is called an uncertain probability space.
\end{definition}

\begin{definition} 
If $\left(\mathcal{U},\mathfrak{U}\right)$ is an arbitrary measurable space, a function $\mathbb{X}$: $\Pi\rightarrow\mathcal{U}$ is a measurable map from $\left(\Pi, \mathbb{F}\right)$ to $\left(\mathcal{U},\mathfrak{U}\right)$ if 

$\mathbb{X}^{-1}(U)\equiv\left\{\boldsymbol{\pi}:\mathbb{X}(\boldsymbol{\pi})\in U\right\}\in\mathbb{F}$ for all $U\in\mathfrak{U}$.

If $(\mathcal{U},\mathfrak{U})=(\mathcal{R},\mathfrak{R})$, $\mathbb{X}$ is called an uncertain random variable.
\end{definition}

\begin{definition}
The uncertain distribution of an uncertain random variable is the function $\mathfrak{F}_{\mathbb{X}}:\ \mathcal{R}\rightarrow\left[0, 1\right]$ defined by
\begin{equation}
\begin{aligned}
\mathfrak{F}_{\mathbb{X}}\left(\mathrm{x}\right)=\mathbb{P}\left(\mathfrak{x}\le\mathrm{x}\right),\mathrm{x}\in\mathcal{R}.
\end{aligned}
\end{equation}

When $\mathfrak{F}_\mathbb{X}\left(\mathrm{x}\right)$ is displayed as
\begin{equation}
\begin{aligned}
\mathfrak{F}_\mathbb{X}\left(\mathrm{x}\right)=\int_{-\infty}^{\mathrm{x}}{\mathfrak{f}_{\mathbb{X}}\left(\mathfrak{x}\right)}d\mathfrak{x},
\end{aligned}
\end{equation}

we say that $\mathfrak{f}_{\mathbb{X}}$ is the uncertain probability density function (UPDF) of $\mathbb{X}$.
\end{definition}

\begin{definition}
The expectation of a continuous uncertain random variable $\mathbb{X}$ is defined by
\begin{equation}
\begin{aligned}
\mu_\mathbb{X}=E(\mathbb{X})=\int_{-\infty}^{\infty}\mathfrak{x}\mathfrak{f}_{\mathbb{X}}(\mathfrak{x})d\mathfrak{x}.
\end{aligned}
\end{equation}

Let $\mathbb{X}$, $\mathbb{Y}$ be uncertain random variables and $E\left|\mathbb{X}\right|$, $E\left|\mathbb{Y}\right|<\infty$, then

\begin{longlist}
\item 
$E\left(\mathbb{X}+\mathbb{Y}\right)=E\left(\mathbb{X}\right)+E\left(\mathbb{Y}\right)$;
\item 
$E\left(a\mathbb{X}+b\right)=aE\left(\mathbb{X}\right)+b$ for any real numbers $a$, $b$;
\item 
$E\left(\mathbb{X}\right)\geq E\left(\mathbb{Y}\right)$, if $\mathbb{X}\geq\mathbb{Y}$.
\end{longlist}
\end{definition}

\begin{definition}
The variance of an uncertain random variable $\mathbb{X}$ is defined by
\begin{equation}
\begin{aligned}
\sigma_{\mathbb{X}}^2=Var\left(\mathbb{X}\right)=E\left[\left(\mathbb{X}-E(\mathbb{X})\right)^2\right].
\end{aligned}
\end{equation}

Its nonnegative square root denoted by $\sigma_{\mathbb{X}}$ is the standard deviation of $\mathbb{X}$. If $z$ is a positive integer, then $E\left(\mathbb{X}^z\right)$ is called the $z$th moment of $\mathbb{X}$.
\end{definition}

\section{The theory and model}

To sum up once more, our problem is that establishing the premises required by Kolmogorov \cite{R14} by "assuming a complex of conditions, $\mathfrak{S}$, which allows of any number of repetitions" is frequently distant from reality. Meanwhile, it is difficult to find exact and convincing correspondence in practice for the presumed upper and lower bounds in the nonlinear expectation theories or it is hard for the nonlinear probability theories to scale up to an infinite number of heterogeneous experiments. Thus, in the following, we begin by outlining the main facts regarding most real random experiments.

\begin{longlist}
\item The conditions differ and thus distinguish experiments.

As is generally known, probability problems may involve a total of $w$ objects of study. When $w>1$, the noticeable differences among them may render the most popular method of conducting experiments with individual objects inappropriate. And only in special cases, do nonlinear theories seem applicable, though remains incapable of overcoming the aforementioned theoretical limitations. If $w>1$, common examples are similar to those agreed on by both Knight \cite{R11} and von Mises \cite{R21} when explaining why the frequency interpretation is not applicable to economics, or even to social science. Moreover, comparable problems often arise in natural science.

If w=1, typical examples include a coin, a die, etc. Here, this one item is the study object of a series of experiments. The ideal experimental environment for KSP is known to exist if the fundamental and influencing conditions with the object remain unchanged. However, the premises required for KSP are not established if the object undergoes fundamental changes and/or varying influences that affect the outcomes of the experiments. In this phase, the nonlinear theories capture changes in the probability distribution at the observable level, rather than at more profound levels that affect the outcomes of experiments. Since they focus on the phenomenal level, a makeshift solution is to approach the problem by making unrealistic assumptions such as bounds.

\item The fact that conditions constantly change adds to the complexity of the problem. 

\item Experiments are also distinguished by the ways that conditions act on the outcomes, which varies depending on the conditions and/or research objects and may also constantly change.

Then, there arises the fascinating question of whether these multiple changes taking place simultaneously exhibit any noteworthy patterns or statistical determinacy.

\end{longlist}

\subsection{Details about partially distinguishable random experiments}

To solve the above problem, we must provide more details about what the prerequisites for PDREs are, and what the stimulating role of "partially distinguishable" is going to be.

\begin{longlist}
\item 
Compared to the repeated experiments in KSP, PDREs only have similarities with one another, which is more true-to-life. 
\item 
Compared to the nonlinear theories, TUP delves deeper into the particular conditions of random experiments. The conditions or factors, here denoted by $X$, that act on the outcomes may distinguish one experiment from another.
\item 
There may be a considerable number of conditions that affect $\varphi$ and thus $V_{\gamma}$, the random variable on $\Phi_{\gamma}$.
\item 
The conditions could generally be classified as either fundamental or influencing, denoted by $X_{m}$ and $X_{l}$ respectively. \footnote{Certain conditions possess both fundamental and influencing characteristics. Psychology is a typical illustration. These conditions can be further divided into two categories of more precise conditions to be handled. When psychology is taken as an example, social components are influencing factors or conditions, while genetic components are fundamental causes or conditions.}
\item 
$X_{m}$ is causing factor of the contingency/randomness or substantial to the distribution of $V_{\gamma}$. \footnote{Especially when fundamental conditions are also intrinsic conditions, as is the case in many situations.}

Since the fundamental conditions are usually present at physical and physiological level where KSP is normally applicable, it is logical to presume that the conditions of most fundamental conditions are difficult to distinguish among all the research objects in a PDRE if $w>1$ or among PDREs if $w=1$. \footnote{As also mentioned later, the conditions of $X_{m}$, the conditions of these conditions, and so forth, may extend down to the level of substance formation, where probability distributions and indistinguishability are predominant. Moreover, we know that it is extremely challenging for us to accurately understand the true differences among the fundamental conditions, even in social science where communication is available, due to the influence of individuals’ subjective judgement, expression, etc.} Thus, it can be reasonably inferred that $X_{m}$ may be random variable defined on $\left(\Omega,F,P\right)$ of KSP and follow certain distributions, and thus denoted by $X_{mi}$ with $i$ representing indistinguishable since, even though we know that $X_{mi}$ may follow a certain distribution, it's still hard to determine its value of any individual research object in each experiment when $w>1$ nor its value in any experiment when $w=1$. 

Although most of the fundamental conditions are indistinguishable in the abovementioned way, we may not completely rule out any possibility of distinguishable fundamental conditions, denoted by $X_{md}$ with $d$ representing distinguishable.
\item 
$X_{l}$ is influencing condition that impacts on the distribution of $V_{\gamma}$.

The conditions of $X_{l}$ may be indistinguishable or incomprehensible as $X_{mi}$, and thus $X_{l}$ could be random variable, represented by $X_{li}$. $X_{ld}$ is influencing and distinguishable condition. \footnote{We can further explore whether $X_{mih}$, $h=1, 2, \ldots, H$, $X_{mdj}$, $j=1, 2, \ldots, J$, $X_{liq}$, $q=1, 2, \ldots, Q$, $X_{likr}$, $r=1, 2, \ldots, R$, $X_{lds}$, $s=1, 2, \ldots, S$, and $X_{ldkt}$, $t=1, 2, \ldots, T$, are independent to facilitate related extensional research, but this paper does not address it for now. Details of the above notation are given in the following text.}
\item 
Let $C_{mi}$ and $C_{md}$ indicate how $X_{mi}$ and $X_{md}$ impact $V_{\gamma}$. Correspondingly, $C_{li}$ and $C_{ld}$ signify the means how $X_{li}$ and $X_{ld}$ influence $V_{\gamma}$.
\item 
As is readily deduced, if $X_{li}$ and $X_{ld}$ are also extrinsic conditions, a lot of them might exercise their influence respectively through intrinsic media $K_{i}$ and $K_{d}$, then $E_{ki}$ and $E_{kd}$ which indicate the effects of $K_{i}$ and $K_{d}$ further transmit the changes to $V_{\gamma}$, and then the influences of $X_{li}$ and $X_{ld}$ are signified as $C_{lik}$ and $C_{ldk}$.Whereas, if $X_{li}$ and $X_{ld}$ exert their influence directly, $C_{li}$ and $C_{ld}$ symbolize the influence mechanism.
\item 

Hence, “partially distinguishable” means: the differences among PDREs should not be substantial and thus they share a same group of $X_{m}$, $X_{l}$, $C_{m}$, $C_{l}$, $C_{lk}$, $K$, and $E_{k}$, some of which might have distinguishable differences among experiments, where $C_{m}=C_{mi}\cup C_{md}$, $C_l=C_{li}\cup C_{ld}$, $C_{lk}=C_{lik}\cup C_{ldk}$, $K=K_{i}\cup K_{d}$, and $E_{k}=E_{ki}\cup E_{kd}$.
\end{longlist}

Specifically, Table \ref{tab1} in Appendix \ref{appA} summarizes all possible types of PDREs when $w=1$, and Table \ref{tab2} in Appendix \ref{appA} lists all possible types of PDREs when $w>1$. In the tables, the first digits of the scenario codes in the first columns represent the numbers of research objects in each experiment. When $w=1$, they are equal to 1. When $w>1$, they are set to $W$.

Initially, we present a few simple examples to illustrate several scenarios that are quite straightforward to understand. Flipping a coin is a typical case representing Table \ref{tab1} scenario 1.ii.iiii where the conditions for a fundamental condition, $X_{m}$, such as the uniformity of the coin, are indistinguishable and thus $X_{m}$ might be a random variable or constant in all experiments. $C_{m}$, the way this $X_{m}$ works on the outcomes is also indistinguishable, and the effect of influencing condition $X_{l}$, such as the flow of air (of course, in the absence of cheating), could be considered ineffective or repetitive. A case representing Table \ref{tab2} W.ii.iiii or W.ii.diii could be to observe the movement of pollen in water while trying to control the water temperature (we know that the actual temperature is rarely completely constant). Upon careful analysis, it can be seen that the other scenarios in the tables are all situations that we may encounter in our research and practice.

In particular, it is necessary to clarify why  $C_{md}$, $C_{li}$, $C_{lik}$, $C_{ld}$ and $C_{ldk}$ could be distinguishable or indistinguishable, which means that the subscripts of these $C$s, $C=C_{m}\cup C_
{l}$, indicate the corresponding relation between $C$ and $X$, not the attributes of $C$s themselves. Moreover, it should also be made clear that whether $C_{md}$, $C_{li}$, $C_{lik}$, $C_{ld}$, and $C_{ldk}$ contain distinguishable differences or not has no bearing on how we perceive their functional forms.

First, $C_{md}$ is where we can start. It corresponds to any possible distinguishable fundamental condition $X_{md}$. On the one hand, for different values of the fundamental condition, the functional form may not change, although, of course, the values of the function do. However, its functional form may vary according to different values of the fundamental condition.

Second, in respect of $C_{li}$, it is different from $C_{mi}$ indicating that the indistinguishable fundamental condition $X_{mi}$ matches the equally indistinguishable fundamental mechanism of action. In this instance, when the conditions of $X_{l}$ are indistinguishable, $X_{l}$ may follow a distribution and only one realization of $X_{li}$ exserts impact in each PDRE. Although $X_{li}$ is indistinguishable in different experiments, it may have different effects on different research objects in a PDRE if $w>1$. In contrast, $C_{li}$ may also naturally suggest that $X_{li}$ has an indistinguishable impact on research objects. However, being indistinguishable does not necessarily mean being identical. It may consist of a collection of functions that are similar to each other, or the functions may include a disturbance term. Thus, $C_{li}$ can be distinguishable or indistinguishable. Then, the reasoning behind $C_{lik}$, $C_{ld}$, $C_{ldk}$, $E_{ki}$, and $E_{kd}$ is comparable to that of $C_{md}$ and $C_{li}$. \footnote{Moreover, the subscript of $K_{i}$ also denotes the matching relation between it and $X_{li}$ but not the attribute of it. When $X_{li}$ is indistinguishable and $C_{li}$ is the opposite, one can easily deduce that $K_{i}$ may be distinguishable.}

Finally, it is particularly important to emphasize the following points. We may not rule out the possibility that, under certain circumstances, a fundamental condition and/or its mechanism of action are distinguishable to varying degrees. If so and when $w>1$, it shall not be a determining factor that dominates the outcomes of the experiments. Otherwise, it can be inferred that the probability will not converge, as in the example of milestones given by von Mises \cite{R22}. \footnote{Hence, we have also shed more light on situations in which the classical probability theory is not applicable.}

Remarkably, the above requirement is not a strong presumption but rather a reasonable setting. The relevant arguments are as follows. Due to the properties of fundamental conditions, $X_{md}$ does not exist at all or only exists in rare cases. $X_{m}$ is the fundamental condition of the experiments, and $X_{m}$ has conditions of its own. By extension, this ultimately leads to the fundamental physical or physiological layer of substance formation, where indistinguishability is a crucial characteristic, as we know. And, as evident from the construction logic of the TUP model in the following text, although the logical hierarchy in the article consists of only 4 layers, the number of layers can be increased to allow for further breakdown until an indistinguishable level is reached. Thus, most, if not all, of fundamental conditions may follow certain distributions. This allows us to know that the combined effects of fundamental conditions with indecisive $X_{md}$ could be random variables. That is

\begin{equation}
V_{\gamma} = g(X). \footnote{$g$ is Borel measurable, which will not be restated hereafter.}
\end{equation}

Furthermore, even if certain fundamental mechanisms of action can be distinguished, in essence, the differences are certainly not weighty and do not change the fact that $g(X)$ may be a random variable. Then, although $X_{l}$ can cause a distinguishable difference in $K$, $K$, essentially as fundamental conditions, and its mechanism of action that might similarly be distinguishable, also cannot change the possible status of $g(X)$ as a random variable. The relevant argument will be more complete when we explain in detail how $X_{l}$ will work in Section 4.2.2..

In actuality, we know that we come across these kinds of partially distinguishable random experiments more frequently, and corresponding analysis requires theoretical tools. That is exactly why we are thinking about coming up with TUP.

\subsection{Formation of uncertain probability density function (UPDF)}

The following succinctly describes the logic behind the formation of UPDF. Every experiment may be unique due to constant changes in fundamental, influencing conditions and/or their mechanisms of action. As a result, $\Phi_{n}$ or $\Phi_{\gamma}$ might vary from experiment to experiment. Meanwhile, the probability distribution of the random variable $V_{n}$ on $\Phi_{n}$ or $V_{\gamma}$ on $\Phi_{\gamma}$ also varies. Thus, a realization value of a distinct $V_{n}$ or $V_{\gamma}$ is a realization of $\mathbb{X}$ in each experiment. In what follows, we will focus on the continuous setting to examine the specifics of the formation of UPDF of $\mathbb{X}$.

\subsubsection{Effect of fundamental conditions on random variables of uncertain experiments}

\begin{longlist}
\item[1.]
Effect of $X_{mi}$ on $V_{\gamma}$. 

Specifically, $C_{mih}$, $h=1,2,\ldots,H$, show how $X_{mih}$, $h=1,2,\ldots,H$, impact $V_{\gamma}$. 

\item[2.]
Effect of $X_{md}$ on $V_{\gamma}$.

$C_{mdj}$, $j=1,2,\ldots,J$, specify how $X_{mdj}$, $j=1,2,\ldots,J$, effect $V_{\gamma}$. 

As previously mentioned, if necessary, each $C_{mdj}$ could represent a set of functions, $C_{mdj}=\left\{c_{{mdj}_{1}},\ c_{m{dj}_{2}}, \ldots\right\}$, which may somehow different with each other for different research objects in a PDRE if $w>1$ or among PDREs if $w=1$. 

\item[3.]
Aggregate Effect of $X_m$ on $V_{\gamma}$.

Then, the collective effect of all $X_{mih}$, $h=1,2,\ldots,H$, and $X_{mdj}$, $j=1,2,\ldots,J$, on $V_{\gamma}$ can be regarded as a function below: 

\begin{equation*}
\begin{aligned}
g_{m}\bigl( & C_{mi1}\left(X_{mi1}\right), C_{mi2}\left(X_{mi2}\right),\ldots, C_{miH}\left(X_{miH}\right), \\ 
& C_{md1}\left(X_{md1}\right), C_{md2}\left(X_{md2}\right),\ldots, C_{mdJ}\left(X_{mdJ}\right)\bigl).
\end{aligned}
\end{equation*}

\end{longlist}

\subsubsection{Effect of influencing conditions on random variables of uncertain experiments}

\begin{longlist}
\item[1.]
Effect of $X_{li}$ on $V_{\gamma}$.
\begin{longlist}
\item
Effect of $X_{li}$ directly on $V_{\gamma}$.

$C_{liq}$, $q=1,2, \ldots,Q$, indicate the means by which $X_{liq}$, $q=1,2, \ldots, Q$, directly affect $V_{\gamma}$.

\item 
Effect of $X_{li}$ on $V_{\gamma}$ acting through $K_i$.

Through the mechanism of $C_{likr}$, $r=1,2, \ldots, R$, $X_{likr}$, $r=1,2,\ldots, R$, may have influence on media $K_{ir}$, $r=1,2,\ldots, R$, and $E_{kir}$, $r=1,2, \ldots, R$, demonstrate the additional transitive impact on $V_{\gamma}$.

\item 
Aggregate effect of $X_{li}$ on $V_{\gamma}$.

Then, the total influence of $X_{li}$ on $V_{\gamma}$ is \footnote{In certain circumstances, if one specific $x_{li}$ exerts influence through multiple intermediate factors, then $K_{ir_{u_r}} = K_{ir_{1}}, K_{ir_{2}}, \ldots, K_{ir_{U_{r}}}$, $u=1, 2, \ldots, U$, represent the various media corresponding to $x_{likr}$. $C_{likr_{u_{r}}} = {C_{likr_{1}}, C_{likr_{2}}, \ldots,  C_{likr_{U_{r}}}}$, and $E_{kir_{u_{r}}} = {E_{kir_{1}}, E_{kir_{2}}, \ldots, E_{kir_{U_{r}}}}$ are how the effects operate. Analogously, the intermediate variables and the associated functions of $x_{ldkt}$ acting on $V_{\gamma}$ are $K_{dt_{u_{t}}} = {K_{dt_{1}}, K_{dt_{2}}, \ldots, K_{dt_{U_{t}}}}$, $C_{ldkt_{u_{t}}} = {C_{ldkr_{1}}, C_{ldkr_{2}}, \ldots, C_{ldk_{U_{t}}}}$ and $E_{kdt_{u_{t}}} = {E_{kdt_{1}}, E_{kdt_{2}}, \ldots, E_{kdt_{U_{t}}}}$, $u=1, 2, \ldots, U$, respectively. Moreover, like $C_{li}$ and $C_{lik}$, of course each $C_{likr_{u_{r}}}$ may be a collection of functions since the ways that $X_{lik}$ impacts $K_{ir_{u_{r}}}$ of various research objects in a PDRE if $w>1$ or among PDREs if w=1 are not necessarily to be identical. $E_{kir_{u_{r}}}$, $C_{ldkt_{u_{t}}}$, and $E_{kdt_{u_{t}}}$ are comparable cases.}

\begin{equation*}
\begin{aligned}
g_{li}\bigl( & C_{li1}\left(X_{li1}\right), C_{li2}\left(X_{li2}\right),\ldots, 
C_{liQ}\left(X_{liQ}\right), E_{ki1}\left(C_{lik1}\left(X_{lik1}\right)\right), \\
& E_{ki2}\left(C_{lik2}\left(X_{lik2}\right)\right), \ldots, E_{kiR}\left(C_{likR}\left(X_{likR}\right)\right)\bigl).
\end{aligned}
\end{equation*}

However, actually, only one realization of each $X_{li}$, $x_{li}\in X_{li}$, exercises effect on $V_{\gamma}$ in each PDRE. Subsequently, the impact of $X_{li}$ on $V_{\gamma}$ is as follows:
\begin{equation}
\begin{aligned}
g_{li}\bigl( & C_{li1}\left(x_{li1}\right), C_{li2}\left(x_{li2}\right),\ldots, 
C_{liQ}\left(x_{liQ}\right), E_{ki1}\left(C_{lik1}\left(x_{lik1}\right)\right), \\
& E_{ki2}\left(C_{lik2}\left(x_{lik2}\right)\right), \ldots, E_{kiR}\left(C_{likR}\left(x_{likR}\right)\right)\bigl).
\end{aligned}
\end{equation}

What is interesting to see is the ultimate impact of $X_{li}$ on the distribution of $V_{\gamma}$ and furthermore that of $\mathbb{X}$. We will particularly examine the situation when $X_{li}$ is normally distributed. 
\end{longlist}

\item[2.]
Effect of $X_{ld}$ on $V_{\gamma}$.
\begin{longlist}
\item 
Effect of $X_{ld}$ directly on $V_{\gamma}$.

$X_{lds}$, $s=1,2, \ldots, S$, could also directly affect $V_{\gamma}$ through the mechanism of $C_{lds}$, $s=1,2, \ldots, S$.

\item 
Effect of $X_{ld}$ on $V_{\gamma}$ acting through $K_d$.
 
Correspondingly, $C_{ldkt}$, $t=1,2, \ldots, T$, signify how $X_{ldkt}$, $t=1,2, \ldots, T$, influence $K_{dt}$, $t=1,2, \ldots, T$. And $E_{kdt}$, $t=1,2, \ldots, T$, further transmit the changes of $K_{dt}$, $t=1,2, \ldots, T$, to $V_{\gamma}$. 

\item 
Aggregate effect of $X_{ld}$ on $V_{\gamma}$.

Likewise, the following illustrates how a single value of every $X_{ld}$ effects $V_{\gamma}$. 

\begin{equation}
\begin{aligned}
g_{ld}\bigl( & C_{ld1}\left(x_{ld1}\right), C_{ld2}\left(x_{ld2}\right), \ldots,
C_{ldS}\left(x_{ldS}\right), E_{kd1}\left(C_{ldk1}\left(x_{ldk1}\right)\right),\\
& E_{kd2}\left(C_{ldk2}\left(x_{ldk2}\right)\right), \ldots, E_{kdT}\left(C_{ldkT}\left(x_{ldkT}\right)\right)\bigl).
\end{aligned}
\end{equation}

Later, we will also show how the functional form of $X_{ld}$ alters the distribution of $\mathbb{X}$. 
\end{longlist}

\item[3.] 
Aggregate effect of $X_{l}$ on $V_{\gamma}$.

Consequently, the collective influence $g_{l}$ of $X_{l}$ on $V_{\gamma}$ is the combination of (4.2) and (4.3).

\end{longlist}

\subsubsection{Aggregate Effect of fundamental and influencing conditions on random variables of uncertain experiments}

Thus, (4.1) should be expressed as follows:

\begin{equation*}
V_{\gamma} = g(X_{m}, X_{l}) = g(X_{m}, x_{li}, x_{ld}, x_{lik}, x_{ldk}). \footnote{Normally, $X_{m}$ may affect $V_{\gamma}$ in two ways: as a fundamental condition itself, and as a result of influencing conditions. In any particular situation, for instance that $K$ which can be also taken as a fundamental condition only acts on $V_{\gamma}$ in response to $X_{l}$, we might consider the situation as the impact of the corresponding $X_{m}$ equal to 0.}
\end{equation*}

It should also be clear from the preceding discussion that the pattern of impact of $X_{m}$, as a random variable, on $V_{\gamma}$ differs from that of $X_{li}$, $X_{ld}$, $X_{lik}$, and $X_{ldk}$, as specific values in each experiment. Explicitly, $X_{m}$ is the primary determinant that establishes the probability distribution of $V_{\gamma}$. In addition to other influences on $V_{\gamma}$, a main and simplified effects of conditions $x_{li}$, $x_{ld}$, $x_{lik}$, and $x_{ldk}$ manifest themselves as horizontal movements of $V_{\gamma}$ through different functional mechanisms and to varying degrees. \footnote{Although there are many other scenarios, a most straightforward and prevalent example in this case is that, coincidentally, $X_{m}$ is an intrinsic condition, $X_{l}$ is an extrinsic influencing condition, and $X_{l}$ acts through the media $K$. $X_{m}$ is a random variable whose value varies randomly among different research objects when $w>1$ or accidentally among different experiments when w=1. Meanwhile, only one realization value of $X_{l}$ takes effect in each experiment, regardless of whether $X_{l}$ is a random variable or not.}  More interestingly, the distributional and/or functional characteristics of $X_{l}$ are imparted to the distribution of $\mathbb{X}$ through the act on $V_{\gamma}$. For example, the means of a family of random variables $\left\{V_{\gamma }:\gamma \in \Lambda \right\}$, where each $V_{\gamma}$ is defined on the corresponding set $\Phi_{\gamma}$, might normally show a cluster close to $\mu$, if $X_{li}\sim N(\mu, \sigma^2)$. At the same time, we can also investigate how the specific form of $X_{ld}$ affects $\mathbb{X}$. \footnote{Although Einstein's \cite{R8} relative statement pertains to the microscopic world where conditions may be incompatible with the scenarios of general life, it indubitably follows that the distribution of distinguishable variable can be identified since our knowledge of it far outweighs its frequency of occurrence, which applies to $X_{ld}$ as well.} 

\subsubsection{Random variable of uncertain experiment and uncertain random variable}

Each $\mathfrak{x}$, a realization of $\mathbb{X}$, uncertainly originates from a distinct $V_{\gamma}$. That is, one and only one value of every $V_{\gamma}$ uncertainly becomes a realization of $\mathbb{X}$ in each experiment, and each $\mathfrak{x}\in\mathbb{X}$ uncertainly comes from a unique distribution of every experiment. In fact, in a world that is constantly changing, this is exactly what we actually experience. 

\subsubsection{Uncertain probability density function (UPDF)}

\begin{lemma}\label{le}
Exchangeability between random variable of uncertain experiment and uncertain random variable.

\normalfont Since uncertain random variable $\mathbb{X}$ is constructed by uncertainly selecting from any possible $V_{\gamma}$, it indicates that $\mathfrak{x}$ and $v_{\gamma}$ are exchangeable. 
\end{lemma}

\begin{theorem}[UPDF]\label{th}
Hence, we have
\begin{equation}
\begin{aligned}
\mathfrak{F}_\mathbb{X}\left(\mathrm{x}\right) = & \mathbb{P}\left(\mathfrak{x}\leq\mathrm{x}\right) = \int_{-\infty}^{\mathrm{x}}{{\mathfrak{f}_{\mathbb{X}}\left(\mathfrak{x}\right)}}d\mathfrak{x} \\
= & \int_{-\infty}^{v_{\gamma}}\int_{-\infty}^{+\infty}\int_{-\infty}^{+\infty}\int_{-\infty}^{+\infty}\int_{-\infty}^{+\infty} \\
& f_{V_{\gamma}, X_{li}, X_{ld}, {X}_{lik}, X_{ldk}}\left(v_{\gamma}, x_{li}, x_{ld}, x_{lik}, x_{ldk}\right)dx_{li}dx_{ld}dx_{lik}dx_{ldk}dv_{\gamma}.
\end{aligned}
\end{equation}

Then,
\begin{equation}
\begin{aligned}
\mathfrak{f}_{\mathbb{X}}(\mathfrak{x}) = \mathfrak{F}_{\mathbb{X}}^\prime
= & \int_{-\infty}^{+\infty}\int_{-\infty}^{+\infty}\int_{-\infty}^{+\infty}\int_{-\infty}^{+\infty} \\
& f_{V_{\gamma}, X_{li}, X_{ld}, {X}_{lik}, X_{ldk}}\left(v_{\gamma}, x_{li}, x_{ld}, x_{lik}, x_{ldk}\right)dx_{li}dx_{ld}dx_{lik}dx_{ldk}.
\end{aligned}
\end{equation}
\end{theorem}

Thus, as we know, $\mathbb{P}\left(\mathbb{X}\right)=1$. 

\section{The valuable and interesting properties of uncertain random variable and its UPDF}

The dynamic UPDF has an excellent capacity to characterize, represent, and interpret real data. We will only discuss its pioneering explanation of typical distributional characteristics and incorporation of pathwise property into distribution model hereafter. \footnote{Relevant empirical analysis will be presented in coming studies.} Prior to this, the following first discusses a number of special cases to illustrate some useful and intriguing properties of UPDF. 

\subsection{TUP and KSP}
\begin{longlist}
\item[1.]
If there is a dominant, fundamental condition governed by some distribution of KSP, TUP degenerates to KSP and may fall into the following two categories. 
\begin{longlist}
\item 
If $w=1$, which means there is only one research object in each experiment, rolling a hard-textured die can serve as a representation of this category of cases.
\item 
If $w>1$, meaning there are more than one research objects in each experiment, observing the diffusion of ink in water at so-called constant temperature is a common example of this category. 
\end{longlist}
\item[2.]
If there are a number of fundamental and influencing conditions, we know that TUP can also degenerate to KSP only if all conditions and $V_{n}$ or $V_{\gamma}$ are repetitive exactly as Kolmogorov required.
\end{longlist}

\subsection{TUP and nonlinear theories}

The following analysis will examine how TUP, in a more realistic and logically compelling manner, acquires the special explanatory power of Choquet theory, $g$-expectation, $G$-expectation and nonlinear probability. 

TUP can also characterize the probabilistic property in Choquet theory which holds significant practical value. Taking the Classic Single-Urn Three-Color Version of the Ellsberg Paradox Experiment \cite{R6} as an illustration, in the first round, when facing the ambiguity of the probability of drawing a black ball, denoted by $P(B)$, individuals' aversion to ambiguity leads them to generally underestimate $P(B)$. It is easy to accept that the degree of ambiguity aversion of a population is a distribution that varies from person to person, not a constant. Thus, a $V_{n}$ of $P(B)$ shifts to the left. Therefore, individuals choose R with known $P(R)$, the probability of red balls, implying $P(B) < P(R)$. In the second round, since $P(B)$ has nothing to do with ambiguity, the corresponding $V_{n}$ does not move negatively, whereas $P(A)$ does. Hence, BY is chosen and $P(B) > P(R)$ could be inferred. Thus, TUP can incorporate various (known) factors that affect probability---here ambiguity aversion into the model.

Specifically, when $\left\{V_{\underline{\mu}}, V_{\overline{\mu}}\right\}\subset\left\{V_{\gamma}:\gamma\in\Lambda\right\}$, we know that this is the topic of concern in the nonlinear expectation theories, which, in fact, analyzes special cases in TUP. The relevant theories developed by Peng (e.g., \cite{R5, R15}) for the first time include uncertainty analysis in the probability distribution. Nonetheless, since they do not delve into the origin and formation of the distribution of $V_{\gamma}$, pertinent inference and problem-solving can only be conducted using rigid assumptions of upper and lower bounds. From the perspective of the theoretical framework of TUP, it is fairly simple to understand why it is challenging to define and estimate the bounds since they actually do not exist. Moreover, although artificial upper and lower bounds allow nonlinear theories to analyze uncertainty, they can only analyze incomplete uncertainty within the artificial limits. Meanwhile, although the nonlinear theories are relatively systematic, including the limit theorem, stochastic analysis, etc., the fit between their models and the real data is far from satisfactory. And there are no compelling logics when the theories are used to describe leptokurtosis, fat tails, asymmetries, etc. However, the explanatory power of probability theories will be considerably enhanced if we can tap into these characteristics with the logic of TUP, as shown below. 

Additionally, researchers who investigate nonlinear probability theories and their applications using multi-armed bandits and two-armed slot machines (e.g., \cite{R2, R4}) have moved beyond the assumption of upper and lower bounds and have pioneered the derivation of nonlinear probability distributions under specific conditions. However, the research remains an analysis of special cases in discrete settings, based on repeated sampling on finite heterogeneous sample spaces and experiments.

\subsection{TUP's explanation of distributional characteristics including those transformed from pathwise property} 

The following will be a comprehensive analysis of common and important distributional characteristics, such as heavy tails/leptokurtosis and asymmetry, including jumps converted from pathwise property. As we know, researchers have put forth various PDFs to improve the models' fit to the actual data by changing the number of parameters and, consequently, the shape of the probability density distribution (e.g., \cite{R17, R1, R13, R19, R24}). Nevertheless, they have never been able to combine heavy tails/leptokurtosis, asymmetries, and especially jumps which are usually investigated in the stochastic process into a single PDF and, of particular importance, explore mechanistic causes for their formation. Unlike previous studies, TUP examines how different distributional characteristics evolve and why some of them are unavoidable.

\subsubsection{Source of the recurrent nature of heavy tails}

Heavy tails, ubiquitous across an exceptionally wide range of empirical domains, raise a foundational methodological puzzle: does their prevalence reflect an ontological property of complex systems, or is it merely a spurious artifact of time-varying volatility and unobserved heterogeneity? Regardless of which explanation one subscribes to, the analysis that follows may offer meaningful insights into the nature of this ubiquity. It is normal for a $X_{li}$ or the aggregate of $X_{liq}$, $q = 1, 2, ..., Q$, and/or $X_{likr}$, $r = 1, 2, ..., R$, to follow a normal distribution $N(\mu,\sigma^2)$ and at the same time there are no distinguishable influencing conditions. Meanwhile, it may be common for $V_\gamma$ to be normal distributed with mean $\mu_{V_\gamma}$ and variance $\sigma_{V_\gamma}^2$. As previously stated, only one realization of $X_l$ affects $V_\gamma$ in each experiment, and $X_l$ affects $V_\gamma$ in several ways, the most significant of which is the horizontal movement of $V_\gamma$. Thus, it is practical for the $\mu_{V_\gamma}$ of a family of random variables $\left\{V_{\gamma }:\gamma \in \Lambda \right\}$, with $V_{\gamma}$ on $\Phi_{\gamma}$, to be a random variable $\mathfrak{M}$ and follow a normal distribution with mean $\mu$ and variance $\sigma^2$. Then, the distribution of $\mathbb{X}$ has heavy tails. The relevant proof is in Appendix \ref{appB}. Moreover, It is easy to see that no matter whether normal distributions are normal or, as Lyon \cite{R12} argues, log normal distributions are normal, the analysis of the frequent occurrence of heavy tails does not differ significantly.

\subsubsection{Formation mechanisms of asymmetries}

It is evident that the formation of asymmetries and/or skewness can be easily demonstrated mathematically by taking the assumption of specific distributions of $V_{\gamma}$ and $X_{l}$. However, what is more worthy of study than the mathematical proof is the mechanism of forming asymmetric characteristics and its practical relevance. The following example focuses on the distribution of financial asset price fluctuations, which is crucial for finance and economics. We know that changes in asset prices are influenced by a variety of factors. International, macroeconomic and industrial factors, as well as the corporate and financial markets themselves, may all act as influencing conditions, $X_{li}$, $X_{ld}$, $X_{lik}$, and $X_{ldk}$. Fundamental conditions, $X_{mi}$ and $X_{md}$, can include investor rationality, sentiment, etc. Moreover, the $X_{l}$ that could result in losses or profits may adhere to the value function proposed by Tversky and Kahneman \cite{R20} as part of their cumulative prospect theory, for which they won a Nobel Prize. Accordingly, people are significantly more sensitive to losses, especially small losses, than to gains of the same magnitude where value function $v\left(x\right)=x^\alpha$ for gain, $x\geq0$, and $v\left(x\right)=-\lambda(-x)^\alpha$ for loss, $x<0$, for $\alpha=0.5$ and $\lambda=2.5$. Specifically, this indicates that $|C_{l}^\prime\left(-X_{l}\right)|>|C_{l}^\prime\left(X_{l}\right)| and |E_{k}^\prime\left(-K\right)|>|E_{k}^\prime\left(K\right)|$, with $-X_{l}$ and $X_{l}$ signifying conditions that involve losses and gains, respectively, and $|E_{k}^\prime\left(-K\right)|$ and $|E_{k}^\prime\left(K\right)|$ indicate short and long based on risk aversion and gain seeking. This in turn entails $|-\Delta\mu_{V_{\gamma}}|>|\Delta\mu_{V_{\gamma}}|$. Consequently, the density $\mathfrak{f}_\mathbb{X}(\mathfrak{x})$ may have its mode located in the negative domain. Moreover, financial theory and behavioral finance can both explain different distributional characteristics, such as asymmetries between the left and right of zero and tails. Thus, economic principles are integrated into the analysis of the formation mechanism and characteristics of the probability distribution. \footnote{Please refer to the follow-up study for a thorough theoretical and empirical analysis.}

\subsubsection{Generating mechanisms of jumps}

As we know, jumps are crucial to risk management, and underestimating the left tail has led to numerous momentous financial events, including the subprime crisis. However, current models have difficulties incorporating and examining jumps along with other critical distributional characteristics. Unlike stochastic analysis and stochastic processes, which essentially analyze jumps at the phenomenal level, TUP integrates jumps into the probability density distribution and provides an explanation of the principles of their formation. 

Regardless of the pattern followed by $X_{l}$, when a specific value, such as a notable negative signal, is reached, $|C_{l}^\prime|$ and/or $|E_{k}^\prime|$ become considerably larger. As previously mentioned, they have various effects on  $V_{\gamma}$ and the most substantial impact is the large-scale horizontal shift of $V_{\gamma}$ in the negative direction. Simultaneously, path dependence is inevitably a result until noteworthy bull news appears. That is, the undesirable changes in $V_{\gamma}$ as asset price fluctuations will again act as market feedback and affect other $V_{\gamma}$ and thus $\mathbb{X}$, subsequently gathering a set of distributions around a position relatively far from the origin on the negative axis. Then, it can be inferred that the randomly selected $\mathbb{X}$ from $\left\{V_\gamma:\gamma\in\Lambda\right\}$ will form some small ridgy distribution such as a jump in a local area in the tail of the distribution. 

\subsection{Other benefits of TUP}

First, it can be easily inferred from the logical reasoning of TUP that it can naturally incorporate the volatility smile and other features of real data into its modeling. Second, it is also easy to see that taking a time series approach to the modeling of TUP can be highly beneficial. For instance, the enhancement of the regression model achieved by Zhu and Müller \cite{R23} might also be further optimized. Third, adding weights to $V_{\gamma}$ may more effectively fulfill the requirements of certain applications. Fourth, an analysis based on increments in certain empirical studies not only complies with reality, but also aids in removing the dimensional differences between various conditions and reveals more intriguing patterns of distribution. Fifth, further investigation will provide more insight into the principles that govern the joint or independent influence pattern of means and variances of a family of random variables $\left\{V_{\gamma }:\gamma \in \Lambda \right\}$ on $\mathbb{X}$. Then, it is obvious that TUP may help explain why the distribution curve might not be that smooth and normative if it is no longer from a mathematical model that is entirely abstract. Next, it offers another approach to modeling uncertainty that is close to the objective causes of uncertainty. Additionally, TUP can incorporate domain-specific knowledge from various fields into probability distributions, thus better meeting the unique application requirements of different domains. Moreover, it will definitely enhance and extend the efficacy of probabilistic methodologies widely employed in artificial intelligence. Finally, regardless of the debate about subjective probability, TUP provides a new way to organically combine objective reality and subjective judgment in modeling, measuring, and analyzing probability in the social sciences, and thus better tackles the problems of the inapplicability of KSP and nonlinear theories to the social sciences and the over-broad attribution of beliefs, which Suárez \cite{R16} argues undermines the objectivity of probability judgments. \footnote{The empirical applications of TUP in areas including finance, economics, education, corporate management, and public administration are discussed in detail in subsequent studies.}

\section{Conclusion}

This paper proposes a new theory of probability---TUP that addresses the problem of overly abstract assumptions of current probability theories, such as repetitive conditions and bounds, and provides a more effective analytical tool for continuously changing experiments that are more common in reality. We refine the basic definition of sample space to accommodate varying experiments, since the world is ever-changing and corresponding experiments also have some degree of distinguishability. We also derive the logic and process of the formation of an innovative kind of probability density distributions based on the ongoing sample spaces in the context of uncertain coaction of the fundamental and influencing conditions of experiments. Then, the theoretical and practical benefits of TUP are briefly discussed. 

In theory, TUP reveals the statistical determinacy of chance events and the greater complexity of deterministic phenomena. Specifically, the coaction of many stochastic phenomena results in distributional characteristics of certainty, and the emergence of other common characteristics is more intricate and involves combined action of deterministic laws and numerous randomness. In other words, a large number of distinct experiments may also present specific statistical regularities. Moreover, TUP helps do complete justice to the interpretations of probability by developing what Hájek \cite{R9} refers to as something of a patchwork, with partially overlapping pieces and principles about how the different interpretations ought to relate. We also touch on the other potential values of the new theoretical framework, such as analyzing uncertainty from a formation mechanism perspective, enhancing the objectivity of subjective probability, and solving the difficult problem of the incompatibility of known and unknown elements in probability theory. In addition, it can be demonstrated that, within the theoretical framework presented in this paper, theories such as Kolmogorov's system of probability and nonlinear theories are special cases. 

In application, it can be anticipated that better application outcomes will result from the actual data being better represented and explained. As William Feller \cite{R7} notes, the modern theory of probability has been claimed to be too abstract to be useful, and new theories are needed to better address the demands of various application scenarios. Exactly, TUP gives the natural sciences more accurate analytical tools, since it approaches reality at the theoretical assumption level. In the social sciences, it is a proper alternative to measure probability in a manner that can logically and systematically weave the subjective and objective into an indivisible whole. Moreover, it can be deduced from the theoretical framework of TUP, in the context of big data, machine learning techniques may be used to deal with continuously changing data and to more precisely estimate distributions and thus to better support diverse applications.

\newpage
\begin{appendix}
\section{All POSSIBLE PDRES}
\label{appA}

\begin{table}[h]
\caption{PDREs when \(w = 1\)}
\label{tab1}
\begin{tabular}{@{}lrrrrrr@{}}
\hline
\noalign{\vskip 15pt}
\parbox{0.8cm}{\centering Sce-\\nario} & \multicolumn{1}{c}{\parbox{1.8cm}{\centering Cond. of $X_m$ \\ \&  Rel. Not.\footnotemark[1]}}
& \multicolumn{1}{c}{\parbox{1.65cm}{\centering Eff. of $C_{m}$ \\ \& Rel. Not.}} & \multicolumn{1}{c}{\parbox{1.65cm}{\centering Cond. of $X_{l}$ \\ \& Rel. Not.}} & \multicolumn{1}{c}{\parbox{1.65cm}{\centering Eff. of $C_{l}$ \\ \&  Rel. Not.}} & \multicolumn{1}{c}{\parbox{1.65cm}{\centering Eff. of $C_{lk}$ \\ \& Rel. Not.}} & \multicolumn{1}{c}{\parbox{1.65cm}{\centering Eff. of $E_{k}$ \\ \& Rel. Not.}}\\
\noalign{\vskip 15pt}
\cline{4-7}
    \noalign{\vskip 1.2pt}
    \cline{4-7}
    \\[0.6ex]

1.ii.iiii & \parbox[t]{1.65cm}{\raggedleft ind:\footnotemark[2] \\ \(X_{mih}\)\footnotemark[3]} & \parbox[t]{1.65cm}{\raggedleft ind: \\ \(C_{mih}\)} & \parbox[t]{1.65cm}{\raggedleft ind: \\ \(X_{liq}\); \(X_{likr}\)\footnotemark[4]} & \parbox[t]{1.65cm}{\raggedleft ind/ine: \\ \(C_{liq}\)}  & \parbox[t]{1.65cm}{\raggedleft ind/ine: \\ \(C_{likr}\)} & \parbox[t]{1.65cm}{\raggedleft ind/ine: \\ \(E_{kir}\)} \\ \cdashline{6-7}[4.5pt/4.3pt] \\[0.8ex]

1.ii.diii & & & \parbox[t]{1.65cm}{\raggedleft dis: \\ \(X_{lds}\); \(X_{ldkt}\)\footnotemark[5]} & \parbox[t]{1.65cm}{\raggedleft ind: \\ \(C_{lds}\)} & \parbox[t]{1.65cm}{\raggedleft ind: \\ \(C_{ldkt}\)} & \parbox[t]{1.65cm}{\raggedleft ind: \\ \(E_{kdt}\)} \\[5.8ex]

1.ii.diid & & & & & & \parbox[t]{1.65cm}{\raggedleft dis: \\ \(E_{kdt}\)} \\[5.8ex]

1.ii.diiid & & & & & & \parbox[t]{1.65cm}{\raggedleft ind \& dis: \\ \(E_{kdt}\)} \\[5.8ex]

1.ii.didi & & & & & \parbox[t]{1.65cm}{\raggedleft dis: \\ \(C_{ldkt}\)} & \parbox[t]{1.65cm}{\raggedleft ind: \\ \(E_{kdt}\)} \\[5.8ex]

1.ii.didd & & & & & & \parbox[t]{1.65cm}{\raggedleft dis: \\ \(E_{kdt}\)} \\[5.8ex]

1.ii.didid & & & & & & \parbox[t]{1.65cm}{\raggedleft ind \& dis: \\ \(E_{kdt}\)} \\[5.8ex]

1.ii.diidi & & & & & \parbox[t]{1.65cm}{\raggedleft ind \& dis: \\ \(C_{ldkt}\)} & \parbox[t]{1.65cm}{\raggedleft ind: \\ \(E_{kdt}\)} \\[5.8ex]

1.ii.diidd & & & & & & \parbox[t]{1.65cm}{\raggedleft dis: \\ \(E_{kdt}\)} \\[5.8ex]

1.ii.diidid & & & & & & \parbox[t]{1.65cm}{\raggedleft ind \& dis: \\ \(E_{kdt}\)} \\ \cdashline{6-7}[4.5pt/4.3pt] \\[0.8ex]

1.ii.dd & & & & \parbox[t]{1.65cm}{\raggedleft dis: \\ \(C_{lds}\)} & \multicolumn{2}{r}{ Replicate dashed content } \\[5.8ex]

1.ii.did & & & & \parbox[t]{1.65cm}{\raggedleft ind \& dis: \\ \(C_{lds}\)} & \multicolumn{2}{r}{ Replicate dashed content } \\ \cline{6-7} 

1.ii.idiii & & & \parbox[t]{1.65cm}{\raggedleft ind \& dis: \\ \(X_{liq}\); \(X_{likr}\); \(X_{lds}\); \(X_{ldkt}\)} & \parbox[t]{1.65cm}{\raggedleft ind: \\ \(C_{liq}\); \(C_{lds}\)} & \parbox[t]{1.65cm}{\raggedleft ind: \\ \(C_{likr}\); \(C_{ldkt}\)} & \parbox[t]{1.65cm}{\raggedleft ind: \\ \(E_{kir}\); \(E_{kdt}\)} \\[5.8ex]

1.ii.idiid & & & & & & \parbox[t]{1.65cm}{\raggedleft dis: \\ \(E_{kdt}\)} \\[5.8ex]\\
\hline
\end{tabular}
\end{table}

\footnotetext[1]{cond., rel., and not. serve as abbreviations for conditions, related, and notation, respectively.}
\footnotetext[2]{ind, ine, and dis serve as abbreviations for indistinguishable, ineffective, and distinguishable, respectively.}
\footnotetext[3]{Where $h = 1, 2, \ldots, H$.}
\footnotetext[4]{Where $q = 1, 2, \ldots, Q$, $r = 1, 2, \ldots, R$.}
\footnotetext[5]{Where $s = 1, 2, \ldots, S$, $t = 1, 2, \ldots, T$.}
\footnotetext[6]{It is only possible for \(C_{lds}\), \(s = 1, 2, \ldots, S\), and \(C_{ldkt}\), \(t = 1, 2, \ldots, T\), to be either indistinguishable or distinguishable.}
\footnotetext[7]{We are aware that \(E_{kir}\) and \(E_{kdt}\) can have a variety of ind and dis combinations. Here, we won't list them one by one. Similar circumstances won't be explained further in the table that follows.}
\footnotetext[8]{Where \(j = 1, 2, \ldots, J\).}


\begin{table}
\begin{tabular}{@{}lrrrrrr@{}}
\hline
\noalign{\vskip 15pt}
\parbox{0.8cm}{\centering Sce-\\nario} & \multicolumn{1}{c}{\parbox{1.8cm}{\centering Cond. of $X_m$ \\ \&  Rel. Not.}}
& \multicolumn{1}{c}{\parbox{1.65cm}{\centering Eff. of $C_{m}$ \\ \& Rel. Not.}} & \multicolumn{1}{c}{\parbox{1.65cm}{\centering Cond. of $X_{l}$ \\ \& Rel. Not.}} & \multicolumn{1}{c}{\parbox{1.65cm}{\centering Eff. of $C_{l}$ \\ \&  Rel. Not.}} & \multicolumn{1}{c}{\parbox{1.65cm}{\centering Eff. of $C_{lk}$ \\ \& Rel. Not.}} & \multicolumn{1}{c}{\parbox{1.65cm}{\centering Eff. of $E_{k}$ \\ \& Rel. Not.}}\\
\noalign{\vskip 30pt}

1.ii.idiiid & & & & & & \parbox[t]{1.65cm}{\raggedleft ind \& dis: \\ \(E_{kir}\); \(E_{kdt}\)} \\[5.7ex]

1.ii.ididi & & & & & \parbox[t]{1.65cm}{\raggedleft dis: \\ \(C_{ldkt}\)} & \parbox[t]{1.65cm}{\raggedleft ind: \\ \(E_{kdt}\)} \\[5.7ex]

1.ii.ididd & & & & & & \parbox[t]{1.65cm}{\raggedleft dis: \\ \(E_{kdt}\)} \\[5.7ex]

1.ii.ididid & & & & & & \parbox[t]{1.65cm}{\raggedleft ind \& dis: \\ \(E_{kdt}\)} \\[5.7ex]

1.ii.idiidi & & & & & \parbox[t]{1.65cm}{\raggedleft ind \& dis: \\ \(C_{likr}\); \(C_{ldkt}\)} & \parbox[t]{1.65cm}{\raggedleft ind: \\ \(E_{kir}\); \(E_{kdt}\)} \\[5.7ex]

1.ii.idiidd & & & & & & \parbox[t]{1.65cm}{\raggedleft dis: \\ \(E_{kdt}\)} \\[5.7ex]

1.ii.idiidid & & & & & & \parbox[t]{1.65cm}{\raggedleft ind \& dis: \\ \(E_{kir}\); \(E_{kdt}\)\footnotemark[6]} \\
\cline{6-7}
\\[0.8ex]

1.ii.idd & & & & \parbox[t]{1.65cm}{\raggedleft dis: \\ \(C_{lds}\) \footnotemark[7]} & \multicolumn{2}{r}{ Replicate thin-line content } \\[5.7ex]

1.ii.idid & & & & \parbox[t]{1.65cm}{\raggedleft ind \& dis: \\ \(C_{liq}\) \& \(C_{lds}\)} & \multicolumn{2}{r}{ Replicate thin-line content }\\
\cline{4-7}
\noalign{\vskip 1.2pt}
\cline{4-7}
\\[0.8ex]

1.idi & \parbox[t]{1.65cm}{\raggedleft ind \& dis: \\ \(X_{mih}\); \(X_{mdj}\)\footnotemark[8]} & \parbox[t]{1.65cm}{\raggedleft ind: \\ \(C_{mih}\); \(C_{mdj}\)} & \multicolumn{4}{r}{ Replicate double-line content } \\[5.7ex]

1.idd & & \parbox[t]{1.65cm}{\raggedleft dis: \\ \(C_{mdj}\)} & \multicolumn{4}{r}{ Replicate double-line content } \\[5.7ex]

1.idid & & \parbox[t]{1.65cm}{\raggedleft ind \& dis: \\ \(C_{mih}\); \(C_{mdj}\)} & \multicolumn{4}{r}{ Replicate double-line content }\\
\hline
\end{tabular}
\end{table}


\begin{table}[h]
\caption{PDREs when \(w > 1\)}
\label{tab2}
\begin{tabular}{@{}lrrrrrr@{}}
\hline
\noalign{\vskip 15pt}
\parbox{0.8cm}{\centering Sce-\\nario} & \multicolumn{1}{c}{\parbox{1.8cm}{\centering Cond. of $X_m$ \\ \&  Rel. Not.}}
& \multicolumn{1}{c}{\parbox{1.65cm}{\centering Eff. of $C_{m}$ \\ \& Rel. Not.}} & \multicolumn{1}{c}{\parbox{1.65cm}{\centering Cond. of $X_{l}$ \\ \& Rel. Not.}} & \multicolumn{1}{c}{\parbox{1.65cm}{\centering Eff. of $C_{l}$ \\ \&  Rel. Not.}} & \multicolumn{1}{c}{\parbox{1.65cm}{\centering Eff. of $C_{lk}$ \\ \& Rel. Not.}} & \multicolumn{1}{c}{\parbox{1.65cm}{\centering Eff. of $E_{k}$ \\ \& Rel. Not.}}\\
\noalign{\vskip 15pt}
\cline{4-7}
    \noalign{\vskip 1.2pt}
    \cline{4-7}
    \noalign{\vskip 3.2pt}
    \cdashline{6-7}[4.5pt/4.5pt]
    \\[0.6ex]

W.ii.iiii & \parbox[t]{1.65cm}{\raggedleft ind: \\ \(X_{mih}\) \\ (among research objects)} & \parbox[t]{1.65cm}{\raggedleft ind: \\ \(C_{mih}\)} & \parbox[t]{1.65cm}{\raggedleft ind: \\ \(X_{liq}\); \(X_{likr}\)} & \parbox[t]{1.65cm}{\raggedleft ind/ine: \\ \(C_{liq}\)} & \parbox[t]{1.65cm}{\raggedleft ind/ine: \\ \(C_{likr}\)} & \parbox[t]{1.65cm}{\raggedleft ind/ine: \\ \(E_{kir}\)} \\[5ex]

W.ii.iidi & & & & & \parbox[t]{1.65cm}{\raggedleft dis: \\ \(C_{likr}\) \\ (among research objects)} & \parbox[t]{1.65cm}{\raggedleft ind: \\ \(E_{kir}\)} \\[5ex]
    
W.ii.iidd & & & & & & \parbox[t]{1.65cm}{\raggedleft dis: \\ \(E_{kir}\)} \\[5ex]
    
W.ii.iidid & & & & & & \parbox[t]{1.65cm}{\raggedleft ind \& dis: \\ \(E_{kir}\)} \\[5ex]
    
W.ii.iiidi & & & & & \parbox[t]{1.65cm}{\raggedleft ind \& dis: \\ \(C_{likr}\)} & \parbox[t]{1.65cm}{\raggedleft ind: \\ \(E_{kir}\)} \\[5ex]
    
W.ii.iiidd & & & & & & \parbox[t]{1.65cm}{\raggedleft dis: \\ \(E_{kir}\)} \\[5ex]
    
W.ii.iiidid & & & & & & \parbox[t]{1.65cm}{\raggedleft ind \& dis: \\ \(E_{kir}\)} \\
\cdashline{6-7}[4.5pt/4.5pt]
    \\ [3ex]
    
W.ii.id & & & & \parbox[t]{1.65cm}{\raggedleft dis: \\ \(C_{liq}\)} & \multicolumn{2}{r}{ Replicate dashed content } \\[5ex]
    
W.ii.iid & & & & \parbox[t]{1.65cm}{\raggedleft ind \& dis: \\ \(C_{liq}\)} & \multicolumn{2}{r}{ Replicate dashed content } \\
\cdashline{6-7}[0.4pt/1.8pt]\\ [3ex]
    
W.ii.diii & & & \parbox[t]{1.65cm}{\raggedleft dis: \\ \(X_{lds}\); \(X_{ldkt}\)} & \parbox[t]{1.65cm}{\raggedleft ind: \\ \(C_{lds}\)} & \parbox[t]{1.65cm}{\raggedleft ind: \\ \(C_{ldkt}\)} & \parbox[t]{1.65cm}{\raggedleft ind: \\ \(E_{kdt}\)} \\[5ex]
    
W.ii.diid & & & & & & \parbox[t]{1.65cm}{\raggedleft dis: \\ \(E_{kdt}\)} \\[5ex]
    
W.ii.diiid & & & & & & \parbox[t]{1.65cm}{\raggedleft ind \& dis: \\ \(E_{kdt}\)} \\[5ex]
    
W.ii.didi & & & & & \parbox[t]{1.65cm}{\raggedleft dis: \\ \(C_{ldkt}\)} & \parbox[t]{1.65cm}{\raggedleft ind: \\ \(E_{kdt}\)} \\[5ex]
    
W.ii.didd & & & & & & \parbox[t]{1.65cm} {\raggedleft dis: \\ \(E_{kdt}\)} \\[5ex] 
    
W.ii.didid & & & & & & \parbox[t]{1.65cm}{\raggedleft ind \& dis: \\ \(E_{kdt}\)} \\
\hline
\end{tabular}
\end{table}

\begin{table}[h]
\begin{tabular}{@{}lrrrrrr@{}}
\hline
\noalign{\vskip 20pt}
\parbox{0.8cm}{\centering Sce-\\nario} & \multicolumn{1}{c}{\parbox{1.8cm}{\centering Cond. of $X_m$ \\ \&  Rel. Not.}}
& \multicolumn{1}{c}{\parbox{1.65cm}{\centering Eff. of $C_{m}$ \\ \& Rel. Not.}} & \multicolumn{1}{c}{\parbox{1.65cm}{\centering Cond. of $X_{l}$ \\ \& Rel. Not.}} & \multicolumn{1}{c}{\parbox{1.65cm}{\centering Eff. of $C_{l}$ \\ \&  Rel. Not.}} & \multicolumn{1}{c}{\parbox{1.65cm}{\centering Eff. of $C_{lk}$ \\ \& Rel. Not.}} & \multicolumn{1}{c}{\parbox{1.65cm}{\centering Eff. of $E_{k}$ \\ \& Rel. Not.}}\\
\noalign{\vskip 20pt}
\hline

W.ii.diidi & & & & & \parbox[t]{1.65cm}{\raggedleft ind \& dis: \\ \(C_{ldkt}\)} & \parbox[t]{1.65cm}{\raggedleft ind: \\ \(E_{kdt}\)} \\[4ex]
    W.ii.diidd & & & & & & \parbox[t]{1.65cm}{\raggedleft dis: \\ \(E_{kdt}\)} \\[4ex]
    W.ii.diidid & & & & & & \parbox[t]{1.65cm}{\raggedleft ind \& dis: \\ \(E_{kdt}\)} \\
    \cdashline{6-7}[0.4pt/1.8pt] \\[1ex]
    W.ii.dd & & & & \parbox[t]{1.65cm}{\raggedleft dis: \\ \(C_{lds}\)} & \multicolumn{2}{r}{ Replicate dotted content } \\[4ex]
    W.ii.did & & & & \parbox[t]{1.65cm}{\raggedleft ind \& dis: \\ \(C_{lds}\)} & \multicolumn{2}{r}{ Replicate dotted content } \\
    \cline{6-7} \\[1ex]
    W.ii.idiii & & & \parbox[t]{1.65cm}{\raggedleft ind \& dis: \\ \(X_{liq}\); \(X_{likr}\); \\ \(X_{lds}\); \(X_{ldkt}\)} & \parbox[t]{1.65cm}{\raggedleft ind: \\ \(C_{liq}\); \(C_{lds}\)} & \parbox[t]{1.65cm}{\raggedleft ind: \\ \(C_{likr}\); \(C_{ldkt}\)} & \parbox[t]{1.65cm}{\raggedleft ind: \\ \(E_{kir}\); \(E_{kdt}\)} \\[4ex]
    W.ii.idiid & & & & & & \parbox[t]{1.65cm}{\raggedleft dis: \\ \(E_{kdt}\)} \\[5ex]
    W.ii.idiiid & & & & & & \parbox[t]{1.65cm}{\raggedleft ind \& dis: \\ \(E_{kir}\); \(E_{kdt}\)} \\[4ex]
    W.ii.ididi & & & & & \parbox[t]{1.65cm}{\raggedleft dis: \\ \(C_{likr}\); \(C_{ldkt}\)} & \parbox[t]{1.65cm}{\raggedleft ind: \\ \(E_{kir}\); \(E_{kdt}\)} \\[4ex]
    W.ii.ididd & & & & & & \parbox[t]{1.65cm}{\raggedleft dis: \\ \(E_{kir}\); \(E_{kdt}\)} \\[4ex]
    W.ii.ididid & & & & & & \parbox[t]{1.65cm}{\raggedleft ind \& dis: \\ \(E_{kir}\); \(E_{kdt}\)} \\[4ex]
    W.ii.idiidi & & & & & \parbox[t]{1.65cm}{\raggedleft ind \& dis: \\ \(C_{likr}\); \(C_{ldkt}\)} & \parbox[t]{1.65cm}{\raggedleft ind: \\ \(E_{kir}\); \(E_{kdt}\)} \\[5ex]
    W.ii.idiidd & & & & & & \parbox[t]{1.65cm}{\raggedleft dis: \\ \(E_{kir}\); \(E_{kdt}\)} \\[4ex]
    W.ii.idiidid & & & & & & \parbox[t]{1.65cm}{\raggedleft ind \& dis: \\ \(E_{kir}\); \(E_{kdt}\)} \\
    \cline{6-7} \\[1ex]
    W.ii.idd & & & & \parbox[t]{1.65cm}{\raggedleft dis: \\ \(C_{liq}\); \(C_{lds}\)} & \multicolumn{2}{r}{ Replicate thin-line content } \\[4ex]
    W.ii.idid & & & & \parbox[t]{1.65cm}{\raggedleft ind \& dis: \\ \(C_{liq}\); \(C_{lds}\)} & \multicolumn{2}{r}{ Replicate thin-line content } \\
    \cline{4-7}
    \noalign{\vskip 1.2pt}
    \cline{4-7} \\[1ex]
    W.idi & \parbox[t]{1.65cm}{\raggedleft ind \& dis: \\ \(X_{mih}\); \(X_{mdj}\)} & \parbox[t]{1.65cm}{\raggedleft ind: \\ \(C_{mih}\); \(C_{mdj}\)} & \multicolumn{4}{r}{Replicate double-line content } \\[4em]
    W.idid & & \parbox[t]{1.65cm}{\raggedleft ind \& dis: \\ \(C_{mih}\); \(C_{mdj}\)} & \multicolumn{4}{r}{ Replicate double-line content }\\
\hline
\end{tabular}
\end{table}

\clearpage
\section{PROOF OF THE RECURRENT NATURE OF HEAVY TAILS}
\label{appB}
\begin{proof}
Given the mean of the random variable $V_\gamma$, its distribution function is
\begin{equation}
\begin{aligned}
   & F_{V_\gamma|\mathfrak{M}}(\mathrm{v}_\gamma|\mathfrak{m})=\mathcal{P}_{V_\gamma}\left(v_\gamma\le\mathrm{v}_\gamma|\mu_{V_\gamma}=\mathfrak{m}\right) \\
   & =\int_{\infty}^{\mathrm{v}_\gamma}\frac{f_{V_\gamma,\ \mathfrak{M}}(v_\gamma,\mathfrak{m})}{f_{\mathfrak{M}}\left(\mathfrak{m}\right)}dv_\gamma=\int_{-\infty}^{\mathrm{v}_\gamma}{f_{V_\gamma|\ \mathfrak{M}}\left(v_\gamma|\mathfrak{m}\right)}dv_\gamma.
\end{aligned}
\end{equation}

According to Theorom \ref{th} and Lemma \ref{le}, an uncertain random variable $\mathbb{X}$ is constructed by randomly selecting from a family of random variables $\left\{V_{\gamma }:\gamma \in \Lambda \right\}$, where each $V_{\gamma}$ is defined on the corresponding set $\Phi_{\gamma}$, indicating that $\mathfrak{x}$ and $v_\gamma$ are exchangeable. Then,
\begin{equation}
\begin{aligned}
    \mathfrak{F}_{\mathbb{X}}(\mathrm{x})=\mathbb{P}(\mathfrak{x}\le\mathrm{x})=\int_{-\infty}^{\mathrm{x}}\mathfrak{f}_{\mathbb{X}}(\mathfrak{x})d\mathfrak{x}=\int_{-\infty}^{\mathrm{v}_\gamma}{\int_{-\infty}^{+\infty}{f_{V_\gamma,\mathfrak{M}}\left(v_\gamma,\mu_{V_\gamma}\right)}d\mu_{V_\gamma}}dv_\gamma.
\end{aligned}
\end{equation}

Based on (B1), (B2) can be

\begin{equation}
    \int_{-\infty}^{\mathrm{v}_{\gamma}}{\int_{-\infty}^{+\infty}{f_{V_{\gamma}|\mathfrak{M}}\left(v_{\gamma}|\mu_{V_{\gamma}}\right)}{f_{\mathfrak{M}}(\mu_{V_{\gamma}})d\mu_{V_{\gamma}}}dv_{\gamma}}. 
\end{equation}

Since
\begin{equation*}
    f_{\mathfrak{M}}\left(\mu_{V_{\gamma}}\right)=\frac{1}{\sigma\sqrt{2\pi}}e^{-\frac{\left(\mu_{V_{\gamma}}-\mu\right)^2}{2\sigma^2}}
\end{equation*}

and
\begin{equation*}
    f_{V_\gamma}\left(v_\gamma\right)=\frac{1}{\sigma_{V_\gamma}\sqrt{2\pi}}e^{-\frac{\left(v_\gamma-\mu_{V_\gamma}\right)^2}{2\sigma_{V_{\gamma}}^2}},
\end{equation*}
(B.3) can be shown as
\begin{equation*}
    \int_{-\infty}^{\mathrm{v}_\gamma}{\int_{-\infty}^{+\infty}\frac{1}{\sigma_{V_\gamma}\sqrt{2\pi}}e^{-\frac{\left(v_\gamma-\mu_{V_\gamma}\right)^2}{2\sigma_{V_{\gamma}}^2}}\frac{1}{\sigma\sqrt{2\pi}}e^{-\frac{\left(\mu_{V_\gamma}-\mu\right)^2}{2\sigma^2}}d\mu_{V_\gamma}}dv_\gamma.
\end{equation*}

Thus,
\begin{equation}
\begin{aligned}
    \mathfrak{f}_{\mathbb{X}}(\mathfrak{x})
    & =\mathfrak{F}_{\mathbb{X}}^\prime={\int_{-\infty}^{+\infty}{\frac{1}{\sigma_{V_{\gamma}}\sqrt{2\pi}}e^{-\frac{\left(v_{\gamma}-\mu_{V_{\gamma}}\right)^2}{2\sigma_{V_{\gamma}}^2}}\frac{1}{\sigma\sqrt{2\pi}}e^{-\frac{\left(\mu_{V_{\gamma}}-\mu\right)^2}{2\sigma^2}}}d\mu_{V_{\gamma}}} \\
    & =\int_{-\infty}^{+\infty}{{f_{V_{\gamma}}\left(v_{\gamma}\right)}f_{\mathfrak{M}}\left(\mu_{V_{\gamma}}\right)}d\mu_{V_{\gamma}}.
\end{aligned}
\end{equation}

Subsequently,
\begin{eqnarray*}
    \mu_{\mathbb{X}} & = & E(\mathbb{X}) \\
                     & = & \int_{-\infty}^{+\infty}{\mathfrak{x}\mathfrak{f}_{\mathbb{X}}(\mathfrak{x})}d\mathfrak{x} \\
                     & = & \int_{-\infty}^{+\infty}\mathfrak{x}\int_{-\infty}^{+\infty}{{f_{V_{\gamma}}\left(v_{\gamma}\right)}f_{\mathfrak{M}}\left(\mu_{V_{\gamma}}\right)}d\mu_{V_{\gamma}}d\mathfrak{x} \\
                     & = & \int_{-\infty}^{+\infty}\mathfrak{x}\int_{-\infty}^{+\infty}\frac{1}{\sigma_{V_{\gamma}}\sqrt{2\pi}}e^{-\frac{\left(v_{\gamma}-\mu_{V_{\gamma}}\right)^2}{2\sigma_{V_{\gamma}}^2}}\frac{1}{\sigma\sqrt{2\pi}}e^{-\frac{\left(\mu_{V_{\gamma}}-\mu\right)^2}{2\sigma^2}}d\mu_{V_{\gamma}}d\mathfrak{x} \\
                     & = & \int_{-\infty}^{+\infty}\frac{1}{\sigma\sqrt{2\pi}}e^{-\frac{\left(\mu_{V_{\gamma}}-\mu\right)^2}{2\sigma^2}}\int_{-\infty}^{+\infty}{{v_{\gamma}}\frac{1}{\sigma_{V_\gamma}\sqrt{2\pi}}e^{-\frac{\left(v_{\gamma}-\mu_{V_{\gamma}}\right)^2}{2\sigma_{V_{\gamma}}^2}}d{v_{\gamma}}}d\mu_{V_{\gamma}} \\
                     & = & \int_{-\infty}^{+\infty}{f_{\mathfrak{M}}\left(\mu_{V_{\gamma}}\right)\mu_{V_{\gamma}}}d\mu_{V_{\gamma}} \\
                     & = & E(\mu_{V_{\gamma}}) \\
                     & = & \mu
\end{eqnarray*}
and
\begin{equation*}
\begin{aligned}
        E(\mathbb{X}^2)
        & = \int_{-\infty}^{+\infty}\mathfrak{x}^2\mathfrak{f}_{\mathbb{X}}(\mathfrak{x})d\mathfrak{x} \\
        & = \int_{-\infty}^{+\infty}{v_{\gamma}}^2\int_{-\infty}^{+\infty}f_{V_{\gamma}}\left(v_{\gamma}\right)f_{\mathfrak{M}}\left(\mu_{V_{\gamma}}\right)d\mu_{V_{\gamma}}dv_{\gamma} \\         
        & = \int_{-\infty}^{+\infty}\frac{1}{\sigma\sqrt{2\pi}}e^{-\frac{\left(\mu_{V_{\gamma}}-\mu\right)^2}{2\sigma^2}}\int_{-\infty}^{+\infty}{{v_{\gamma}}^2\frac{1}{\sigma_{V_{\gamma}}\sqrt{2\pi}}e^{-\frac{\left(v_{\gamma}-\mu_{V_{\gamma}}\right)^2}{2\sigma_{V_{\gamma}}^2}}dv_{\gamma}}d\mu_{V_{\gamma}} \\
        & = \int_{-\infty}^{+\infty}{\frac{1}{\sigma\sqrt{2\pi}}e^{-\frac{\left(\mu_{V_{\gamma}}-\mu\right)^2}{2\sigma^2}}\left({\mu_{V_{\gamma}}}^2+\sigma_{V_{\gamma}}^2\right)}d\mu_{V_{\gamma}} \\ 
        & = \sigma^2+\mu^2+\sigma_{V_{\gamma}}^2.
\end{aligned}
\end{equation*}

Then,
\begin{equation*}
\sigma_{\mathbb{X}}^2=Var\left(\mathbb{X}\right)=E(\mathbb{X}^2)-\left(E\left(\mathbb{X}\right)\right)^2=\sigma^2+\sigma_{V_{\gamma}}^2.
\end{equation*}

Consequently, 
\begin{align}
    Kurt\left(\mathbb{X}\right)
    = & E\left[\left(\frac{\mathbb{X}-\mu_{\mathbb{X}}}{\sigma_{\mathbb{X}}}\right)^4\right] \nonumber \\
    = & \int_{-\infty}^{+\infty}\left(\frac{\mathfrak{x}-\mu}{\sigma_{\mathbb{X}}}\right)^4\mathfrak{f}_{\mathbb{X}}\left(\mathfrak{x}\right)d\mathfrak{x} \nonumber \\
    = & \int_{-\infty}^{+\infty}\left(\frac{v_{\gamma}-\mu}{\sigma_{\mathbb{X}}}\right)^4\int_{-\infty}^{+\infty}f_{V_{\gamma}}\left(v_{\gamma}\right)f_{\mathfrak{M}}\left(\mu_{V_{\gamma}}\right)d\mu_{V_\gamma}dv_{\gamma} \nonumber \\
    = & \int_{-\infty}^{+\infty}f_{\mathfrak{M}}\left(\mu_{V_\gamma}\right)\int_{-\infty}^{+\infty}{\left(\frac{v_{\gamma}-\mu}{\sigma_{\mathbb{X}}}\right)^4}{f_{V_{\gamma}}\left(v_{\gamma}\right)dv_{\gamma}}d\mu_{V_{n}} \nonumber \\
    = & \int_{-\infty}^{+\infty}f_{\mathfrak{M}}\left(\mu_{V_{\gamma}}\right)\int_{-\infty}^{+\infty}{\left(\frac{v_{\gamma}-\mu}{\sigma_{\mathbb{X}}}\right)^4\frac{1}{\sigma_{V_{\gamma}}\sqrt{2\pi}}e^{-\frac{\left(v_{\gamma}-\mu_{V_{\gamma}}\right)^2}{2\sigma_{V_{\gamma}}^2}}}dv_{\gamma}d\mu_{V_{\gamma}} \nonumber \\
    = & \int_{-\infty}^{+\infty}f_{\mathfrak{M}}\left(\mu_{V_{\gamma}}\right) \\
    & \int_{-\infty}^{+\infty}{\left(\frac{\sigma_{V_{\gamma}}}{\sigma_{\mathbb{X}}}\frac{v_{\gamma}-\mu_{V_{\gamma}}}{\sigma_{V_{\gamma}}}+\frac{\mu_{V_{\gamma}}-\mu}{\sigma_{\mathbb{X}}}\right)^4\frac{1}{\sigma_{V_{\gamma}}\sqrt{2\pi}}e^{-\frac{\left(v_{\gamma}-\mu_{V_{\gamma}}\right)^2}{2\sigma_{V_{\gamma}^2}}}dv_{\gamma}}d\mu_{V_{\gamma}} \nonumber \\
    = & \int_{-\infty}^{+\infty}f_{\mathfrak{M}}\left(\mu_{V_{\gamma}}\right)[
    3\frac{{\sigma_{V_{\gamma}}}^4}{{\sigma_{\mathbb{X}}}^4}+6\frac{{\sigma_{V_{\gamma}}}^2}{{\sigma_{\mathbb{X}}}^2}\frac{{\left(\mu_{V_{\gamma}}-\mu\right)}^2}{{\sigma_{\mathbb{X}}}^2}+\frac{{\left(\mu_{V_{\gamma}}-\mu\right)}^4}{{\sigma_{\mathbb{X}}}^4}]d\mu_{V_{\gamma}} \nonumber \\
    = & E_{\mu_{V_{\gamma}}}[\frac{{\left(\mu_{V_{\gamma}}-\mu\right)}^4}{{\sigma_{\mathbb{X}}}^4}]+
    \int_{-\infty}^{+\infty}{f_{\mathfrak{M}}\left(\mu_{V_{\gamma}}\right)\left[3\frac{{\sigma_{V_{\gamma}}}^4}{{\sigma_{\mathbb{X}}}^4}+6\frac{{{\sigma_{V_{\gamma}}}^2}{\left(\mu_{V_{\gamma}}-\mu\right)}^2}{{\sigma_{\mathbb{X}}}^4}\right]}d\mu_{V_{\gamma}} \nonumber \\
    = & 3 + \int_{-\infty}^{+\infty}{f_{\mathfrak{M}}\left(\mu_{V_{\gamma}}\right)\left[ 3 \frac{{\sigma_{V_{\gamma}}}^4}{{\sigma_{\mathbb{X}}}^4} + 6\frac{{{\sigma_{V_{\gamma}}}^2}{\left(\mu_{V_{\gamma}}-\mu\right)}^2}{{\sigma_{\mathbb{X}}}^4}\right]}d\mu_{V_{n}}. \nonumber
\end{align}

If only $\sigma_{V_{\gamma}}^2 > 0$,

\begin{equation*}
\begin{aligned}
\left[ 3 \frac{{\sigma_{V_{\gamma}}}^4}{{\sigma_{\mathbb{X}}}^4} + 6\frac{{{\sigma_{V_{\gamma}}}^2}{\left(\mu_{V_{\gamma}}-\mu\right)}^2}{{\sigma_{\mathbb{X}}}^4}\right] > 0.
\end{aligned}
\end{equation*}

If $\sigma_{V_{\gamma}}^2$ remains constant across $\gamma$, (B5) can be further simplified as

\begin{equation*}
\begin{aligned}
    Kurt\left(\mathbb{X}\right)
    & = 3 + 3\frac{{\sigma_{V_{\gamma}}}^4}{{\sigma_{\mathbb{X}}}^4}+6\frac{{\sigma_{V_{\gamma}}}^2}{{\sigma_{\mathbb{X}}}^4}E_{\mu_{V_\gamma}}\left[\left(\mu_{V_{\gamma}}-\mu\right)^2\right] \\
    & =3 + 3\frac{{\sigma_{V_{\gamma}}}^4}{{\sigma_{\mathbb{X}}}^4}+6\frac{{\sigma_{V_{\gamma}}}^4}{{\sigma_{\mathbb{X}}}^4} \\
    & = 3 + 9\frac{{\sigma_{V_{\gamma}}}^4}{{\sigma_{\mathbb{X}}}^4}.
\end{aligned}
\end{equation*}

Since excess kurtosis is positive, the distribution of $\mathbb{X}$ is heavy-tailed.

\end{proof}
\end{appendix}

\begin{acks}[Acknowledgments]
I warmly thank my postdoctoral supervisor Peng Shige. With lively and engaging account of the evolution of probability theories, alongside his personal experiences (partially outlined in \cite{R15}) in developing BSDE and nonlinear expectation, he encouraged me to think from the origin.
\end{acks}

\begin{funding}
The author was supported by NSFC Grant 11971268, SDST KJHZ026, the talent program Taishan Scholar Young Expert.
\end{funding}



\bibliographystyle{imsart-number} 
\bibliography{ref}       


\end{document}